\documentclass[11pt,reqno,oneside]{amsart}
\usepackage[utf8]{inputenc}
\usepackage[T2A]{fontenc}
\usepackage[english]{babel}
\usepackage{amsmath, amssymb, amsthm}
\usepackage[pdfborder={0 0 0}, pdffitwindow = true, pdfstartview = FitH,  pdftitle = {The lower bound for the number of facets of a k-neighborly d-polytope with d+3 vertices},
pdfauthor = {A.N. Maksimenko}
pdfkeywords = {Gale diagram}
]{hyperref}

\usepackage[margin=2.75cm]{geometry}

\usepackage{subcaption}
\usepackage{tikz}
\usetikzlibrary{arrows} 
\usetikzlibrary{patterns}
\usetikzlibrary[calc]
\usetikzlibrary{decorations.pathreplacing} 

\definecolor{lightblue}{rgb}{0.7,0.8,1} 
\tikzstyle{thickr}=[thick, myred]  

\newcommand{\smallradius}{.07}

\newcommand{\point}{circle (\smallradius) }

\newcommand{\bradius}{2}
\newcommand{\Bradius}{2}

\newcommand{\diameter}[1]{\draw [dashed, thin] (#1:\bradius) -- ({180+#1}:\bradius);}

\tikzset{filled/.style={draw=black, fill=black}}
\tikzset{semifilled/.style={draw=black, fill=white}} 
\tikzset{empty/.style={draw=black, densely dotted, fill=white}}

\newcommand{\picRomb}
{
		\begin{tikzpicture}[scale=0.7, fill=black, draw=black] 
			\filldraw [empty, dotted] (0, 0) circle (\bradius);
			\diameter{90}
			\filldraw [filled] (90:\bradius) \point node[anchor=south] {$3$};
			\filldraw [filled] (-90:\bradius) \point node[anchor=north] {$3$};
			\diameter{0}
			\filldraw [filled] (0:\bradius) \point node[anchor=west] {$3$};
			\filldraw [filled] (180:\bradius) \point node[anchor=east] {$3$};
		\end{tikzpicture}
}

\newcommand{\picOnes}
{
		\begin{tikzpicture}[scale=0.7, filled] 
			\filldraw [empty, dotted] (0, 0) circle (\bradius);
			\diameter{135}
			\diameter{0}
			\diameter{-135}
			\diameter{90}
			\filldraw [filled] (0:\bradius) \point node[anchor=west] {$1$};
			\filldraw [filled] (45:\bradius) \point node[anchor=south west] {$1$};
			\filldraw [filled] (90:\bradius) \point node[anchor=south] {$1$};
			\filldraw [filled] (135:\bradius) \point node[anchor=south east] {$1$};
			\filldraw [filled] (180:\bradius) \point node[anchor=east] {$1$};
			\filldraw [filled] (-135:\bradius) \point node[anchor=north east] {$1$};
			\filldraw [filled] (-90:\bradius) \point node[anchor=north] {$1$};
			\filldraw [filled] (-45:\bradius) \point node[anchor=north west] {$1$};
		\end{tikzpicture}
}

\newcommand{\oseventh}{90/7}

\newcommand{\picCyclic}
{
		\begin{tikzpicture}[scale=0.7, filled] 
			\filldraw [empty, dotted] (0, 0) circle (\bradius);
			\diameter{180-5*\oseventh}
			\diameter{-3*\oseventh}
			\diameter{180-\oseventh}
			\diameter{\oseventh}
			\diameter{3*\oseventh-180}
			\diameter{5*\oseventh}
			\diameter{-90}
			\filldraw [empty] (\oseventh:\bradius) \point;
			\filldraw [filled] (3*\oseventh:\bradius) \point node[anchor=south west] {$1$};
			\filldraw [empty] (5*\oseventh:\bradius) \point;
			\filldraw [filled] (90:\bradius) \point node[anchor=south] {$1$};
			\filldraw [empty] (180 - 5*\oseventh:\bradius) \point;
			\filldraw [filled] (180 - 3*\oseventh:\bradius) \point node[anchor=south east] {$1$};
			\filldraw [empty] (180 - \oseventh:\bradius) \point;
			\filldraw [filled] (\oseventh-180:\bradius) \point node[anchor=east] {$1$};
			\filldraw [empty] (3*\oseventh-180:\bradius) \point;
			\filldraw [filled] (5*\oseventh-180:\bradius) \point node[anchor=north east] {$1$};
			\filldraw [empty] (-90:\bradius) \point;
			\filldraw [filled] (-5*\oseventh:\bradius) \point node[anchor=north west] {$1$};
			\filldraw [empty] (-3*\oseventh:\bradius) \point;
			\filldraw [filled] (-\oseventh:\bradius) \point node[anchor=west] {$1$};
		\end{tikzpicture}
}

\newcommand{\picShift}
{
		\begin{tikzpicture}[scale=0.7, fill=black, draw=black] 
			\filldraw [empty, dotted] (0, 0) circle (\bradius);
			\diameter{-120}
			\filldraw [filled] (60:\bradius) \point node[anchor=south west] {$m_i$};
			\draw [-stealth', thick] ({60+7}:\Bradius) arc ({60+7}:{90-7}:\Bradius);
			\filldraw [empty] (180+60:\bradius) \point node[anchor=north east] {$0$};
			\filldraw [semifilled] (120:\bradius) \point node[anchor=south east] {$m_{i-2}$};
			\filldraw [semifilled] (180+120:\bradius) \point node[anchor=north west] {$m_{i+n-2}$};
			\filldraw [semifilled] (30:\bradius) \point node[anchor=south west] {$m_{i+1}$};
			\filldraw [semifilled] (180+30:\bradius) \point node[anchor=north east] {$m_{i+n+1}$};
			\filldraw [draw=white, draw opacity = 0, fill=gray, fill opacity=0.3] %
			(88:{0.1 + \bradius}) arc (90:-90:{0.1 + \bradius})  -- cycle;
			\diameter{90}
			\filldraw [semifilled] (90:\bradius) \point node[anchor=south] {$m_{i-1}$};
			\filldraw [filled] (-90:\bradius) \point node[anchor=north] {$m_{i+n-1}$};
		\end{tikzpicture}
}

\newcommand{\picSimplicialA}
{
		\begin{tikzpicture}[scale=0.7, filled] 
			\filldraw [empty, dotted] (0, 0) circle (\bradius);
			\diameter{180-5*\oseventh}
			\diameter{-3*\oseventh}
			\diameter{180-\oseventh}
			\diameter{\oseventh}
			\diameter{3*\oseventh-180}
			\diameter{5*\oseventh}
			\diameter{-90}
			\filldraw [empty] (\oseventh:\bradius) \point;
			\filldraw [filled] (3*\oseventh:\bradius) \point node[anchor=south west] {$1$};
			\filldraw [empty] (5*\oseventh:\bradius) \point;
			\filldraw [filled] (90:\bradius) \point node[anchor=south] {$2$};
			\filldraw [empty] (180 - 5*\oseventh:\bradius) \point;
			\filldraw [filled] (180 - 3*\oseventh:\bradius) \point node[anchor=south east] {$1$};
			\filldraw [empty] (180 - \oseventh:\bradius) \point;
			\filldraw [filled] (\oseventh-180:\bradius) \point node[anchor=east] {$2$};
			\filldraw [empty] (3*\oseventh-180:\bradius) \point;
			\filldraw [filled] (5*\oseventh-180:\bradius) \point node[anchor=north east] {$1$};
			\filldraw [empty] (-90:\bradius) \point;
			\filldraw [filled] (-5*\oseventh:\bradius) \point node[anchor=north west] {$2$};
			\filldraw [empty] (-3*\oseventh:\bradius) \point;
			\filldraw [filled] (-\oseventh:\bradius) \point node[anchor=west] {$1$};
		\end{tikzpicture}
}

\newcommand{\picSimplicialB}
{
		\begin{tikzpicture}[scale=0.7, filled] 
			\filldraw [empty, dotted] (0, 0) circle (\bradius);
			\diameter{180-5*\oseventh}
			\diameter{-3*\oseventh}
			\diameter{180-\oseventh}
			\diameter{\oseventh}
			\diameter{3*\oseventh-180}
			\diameter{5*\oseventh}
			\diameter{-90}
			\filldraw [filled] (\oseventh:\bradius) \point node[anchor=west] {$1$};
			\filldraw [empty] (3*\oseventh:\bradius) \point;
			\filldraw [filled] (5*\oseventh:\bradius) \point node[anchor=south] {$1$};
			\filldraw [filled] (90:\bradius) \point node[anchor=south] {$1$};
			\filldraw [filled] (180 - 5*\oseventh:\bradius) \point node[anchor=south] {$1$};
			\filldraw [empty] (180 - 3*\oseventh:\bradius) \point;
			\filldraw [filled] (180 - \oseventh:\bradius) \point node[anchor=east] {$1$};
			\filldraw [filled] (\oseventh-180:\bradius) \point node[anchor=east] {$1$};
			\filldraw [filled] (3*\oseventh-180:\bradius) \point node[anchor=north east] {$1$};
			\filldraw [empty] (5*\oseventh-180:\bradius) \point;
			\filldraw [filled] (-90:\bradius) \point node[anchor=north] {$1$};
			\filldraw [filled] (-5*\oseventh:\bradius) \point node[anchor=north west] {$1$};
			\filldraw [filled] (-3*\oseventh:\bradius) \point node[anchor=north west] {$1$};
			\filldraw [empty] (-\oseventh:\bradius) \point;
		\end{tikzpicture}
}

\newcommand{\picMMS}
{
		\begin{tikzpicture}[scale=0.7, filled] 
			\filldraw [empty, dotted] (0, 0) circle (\bradius);
			\diameter{-5*\oseventh}
			\diameter{3*\oseventh}
			\filldraw [empty] (180 - 5*\oseventh:\bradius) \point node[anchor=south] {$j+n$};
			\filldraw [filled] (180 - 2*\oseventh:\bradius) \point node[anchor=east] {$s$};
			\filldraw [empty] (180 + 3*\oseventh:\bradius) \point node[anchor=north east] {$i+n$};
			\filldraw [filled] (-5*\oseventh:\bradius) \point node[anchor=north] {$j$};
			\filldraw [empty] (-90:\bradius) \point node[anchor=north east] {$j+1$};
			\filldraw [empty] (\oseventh:\bradius) \point node[anchor=west] {$i+1$};
			\filldraw [filled] (3*\oseventh:\bradius) \point node[anchor=south west] {$i$};
			\filldraw [draw=white, draw opacity = 0, fill=gray, fill opacity=0.3] %
			(3*\oseventh:\bradius) -- (-5*\oseventh:\bradius) -- (180 - 2*\oseventh:\bradius) -- cycle;
			\draw [-stealth', thick] ({3*\oseventh - 5}:\Bradius) arc ({3*\oseventh - 5}:{1*\oseventh + 5}:\Bradius);
			\draw [-stealth', thick] ({-5*\oseventh - 5}:\Bradius) arc ({-5*\oseventh - 5}:{-90 + 5}:\Bradius);
		\end{tikzpicture}
}

\newcommand{\picConsecutiveA}
{
		\begin{tikzpicture}[scale=0.7, fill=black, draw=black] 
			\filldraw [empty, dotted] (0, 0) circle (\bradius);
			\diameter{60}
			\diameter{90}
			\diameter{120}
			\filldraw [filled] (120:\bradius) \point node[anchor=south east] {$m_{i}$};
			\filldraw [empty] (90:\bradius) \point node[anchor=south] {$0$};
			\filldraw [filled] (60:\bradius) \point node[anchor=south west] {$m_{i+2}$};
			\draw [-stealth', thick] ({60+7}:\Bradius) arc ({60+7}:{90-7}:\Bradius);
			\filldraw [empty] (180+60:\bradius) \point node[anchor=north east] {$0$};
			\filldraw [filled] (-90:\bradius) \point node[anchor=north] {$m_{i+1+n}$};
			\filldraw [filled] (-60:\bradius) \point node[anchor=north west] {$m_{i+n}$};
		\end{tikzpicture}
}

\newcommand{\picConsecutiveB}
{
		\begin{tikzpicture}[scale=0.7, fill=black, draw=black] 
			\filldraw [empty, dotted] (0, 0) circle (\bradius);
			\diameter{60}
			\diameter{90}
			\diameter{120}
			\filldraw [filled] (120:\bradius) \point node[anchor=south east] {$m_{i}$};
			\filldraw [empty] (90:\bradius) \point node[anchor=south] {$0$};
			\filldraw [filled] (60:\bradius) \point node[anchor=south west] {$m_{i+2}$};
			\draw [-stealth', thick] ({120-7}:\Bradius) arc ({120-7}:{90+7}:\Bradius);
			\filldraw [filled] (180+60:\bradius) \point node[anchor=north east] {$m_{i+2+n}$};
			\filldraw [filled] (-90:\bradius) \point node[anchor=north] {$m_{i+1+n}$};
			\filldraw [empty] (-60:\bradius) \point node[anchor=north west] {$0$};
		\end{tikzpicture}
}

\newcommand{\picSumTwoA}
{
		\begin{tikzpicture}[scale=0.7, fill=black, draw=black] 
			\filldraw [empty, dotted] (0, 0) circle (\bradius);
			\diameter{30}
			\diameter{60}
			\diameter{90}
			\filldraw [filled] (90:\bradius) \point node[anchor=south] {$m_{n+i-1}$};
			\filldraw [filled] (60:\bradius) \point node[anchor=south west] {$m_{n+i}$};
			\filldraw [filled] (30:\bradius) \point node[anchor=south west] {$m_{n+i+1}$};
			\filldraw [filled] (-90:\bradius) \point node[anchor=north] {$1$};
			\filldraw [empty] (180+60:\bradius) \point node[anchor=north east] {$0$};
			\filldraw [filled] (180+30:\bradius) \point node[anchor=north east] {$m_{i+1}$};
			\draw [-stealth', thick] ({60+7}:\Bradius) arc ({60+7}:{90-7}:\Bradius);
		\end{tikzpicture}
}

\newcommand{\picSumTwoB}
{
		\begin{tikzpicture}[scale=0.7, fill=black, draw=black] 
			\filldraw [empty, dotted] (0, 0) circle (\bradius);
			\diameter{30}
			\diameter{60}
			\diameter{90}
			\filldraw [filled] (90:\bradius) \point node[anchor=south] {$m_{n+i-2}$};
			\filldraw [filled] (60:\bradius) \point node[anchor=south west] {$m_{n+i-1}$};
			\filldraw [filled] (30:\bradius) \point node[anchor=south west] {$m_{n+i}$};
			\filldraw [filled] (-90:\bradius) \point node[anchor=north] {$m_{i-2}$};
			\filldraw [empty] (180+60:\bradius) \point node[anchor=north east] {$0$};
			\filldraw [filled] (180+30:\bradius) \point node[anchor=north east] {$1$};
			\draw [-stealth', thick] ({60-7}:\Bradius) arc ({60-7}:{30+7}:\Bradius);
		\end{tikzpicture}
}

\newcommand{\picHalfA}
{
		\begin{tikzpicture}[scale=0.7, filled] 
			\filldraw [empty, dotted] (0, 0) circle (\bradius);
			\diameter{135}
			\diameter{0}
			\diameter{-135}
			\diameter{90}
			\filldraw [semifilled] (0:\bradius) \point node[anchor=west] {$m_3$};
			\filldraw [semifilled] (45:\bradius) \point node[anchor=south west] {$m_2$};
			\filldraw [semifilled] (90:\bradius) \point node[anchor=south] {$m_1$};
			\filldraw [empty] (135:\bradius) \point node[anchor=south east] {$0$};
			\filldraw [filled] (180:\bradius) \point node[anchor=east] {$m_7$};
			\filldraw [empty] (-135:\bradius) \point node[anchor=north east] {$0$};
			\filldraw [semifilled] (-90:\bradius) \point node[anchor=north] {$m_5$};
			\filldraw [semifilled] (-45:\bradius) \point node[anchor=north west] {$m_4$};
		\end{tikzpicture}
}

\newcommand{\picHalfB}
{
		\begin{tikzpicture}[scale=0.7, filled] 
			\filldraw [empty, dotted] (0, 0) circle (\bradius);
			\diameter{135}
			\diameter{0}
			\diameter{-135}
			\diameter{90}
			\filldraw [semifilled] (0:\bradius) \point node[anchor=west] {$m_3$};
			\filldraw [filled] (45:\bradius) \point node[anchor=south west] {$m_2$};
			\filldraw [filled] (90:\bradius) \point node[anchor=south] {$m_1$};
			\filldraw [empty] (135:\bradius) \point node[anchor=south east] {$0$};
			\filldraw [filled] (180:\bradius) \point node[anchor=east] {$m_7$};
			\filldraw [filled] (-135:\bradius) \point node[anchor=north east] {$1$};
			\filldraw [filled] (-90:\bradius) \point node[anchor=north] {$m_5$};
			\filldraw [semifilled] (-45:\bradius) \point node[anchor=north west] {$m_4$};
		\end{tikzpicture}
}

\newcommand{\picInduction}
{
		\begin{tikzpicture}[scale=0.7, filled] 
			\filldraw [empty, dotted] (0, 0) circle (\bradius);
			\diameter{90}
			\diameter{60-180}
			\diameter{30}
			\diameter{-30}
			\diameter{180-60}
			\filldraw [semifilled] (90:\bradius) \point node[anchor=south] {$x_1$};
			\filldraw [semifilled] (60:\bradius) \point node[anchor=south west] {$x_2$};
			\filldraw [semifilled] (30:\bradius) \point node[anchor=south west] {$x_3$};
			\filldraw [semifilled] (-30:\bradius) \point node[anchor=north west] {$x_{n-1}$};
			\filldraw [semifilled] (-60:\bradius) \point node[anchor=north west] {$x_n$};
			\filldraw [semifilled] (-90:\bradius) \point node[anchor=north] {$y_1$};
			\filldraw [semifilled] (-120:\bradius) \point node[anchor=north east] {$y_2$};
			\filldraw [semifilled] (-150:\bradius) \point node[anchor=north east] {$y_3$};
			\filldraw [semifilled] (150:\bradius) \point node[anchor=south east] {$y_{n-1}$};
			\filldraw [semifilled] (120:\bradius) \point node[anchor=south east] {$y_{n}$};
		\end{tikzpicture}
}

\usepackage{listings}

\theoremstyle{plain}
\newtheorem{theorem}{Theorem}
\newtheorem{lemma}[theorem]{Lemma}
\newtheorem{proposition}[theorem]{Proposition}

\theoremstyle{definition}
\newtheorem{example}{Example}
\theoremstyle{remark}

\newcommand{\R}{\mathbb{R}}    
\newcommand{\No}{\textnumero}    

\DeclareMathOperator{\ext}{ext}
\DeclareMathOperator{\conv}{conv}

\title{The lower bound for the number of facets \\of a k-neighborly d-polytope with d+3 vertices}

\author{Aleksandr Maksimenko}
\thanks{Supported by the~project \No\,1.5768.2017/П220 of P.\,G.~Demidov Yaroslavl State University within State Assignment for~Research.}
\address{Laboratory of Discrete and Computational Geometry, P.G. Demidov Yaroslavl State University, ul. Sovetskaya 14, Yaroslavl 150000, Russia} 
\email{maximenko.a.n@gmail.com}

\begin{document}

\begin{abstract}
We have found the minimal difference $\Delta(k) = \min\limits_P (f_{d-1}(P) - f_{0}(P))$ between the number of facets and the number of vertices of a~$k$-neigh\-bor\-ly $d$-polytope $P$ for the case $f_{0}(P) = d+3$: $\Delta(2) = 4$, $\Delta(3) = 15$, and~$\Delta(k) = 2 (k^2 - 1)$ for $k \ge 4$.
\end{abstract}

\maketitle

\section{Introduction}

In this paper we consider only convex $d$-dimensional polytopes, i.e. \emph{$d$-polytopes}.
A $d$-polytope $P$ is called \emph{$k$-neigh\-bor\-ly} if every set of $k$ vertices forms the vertex set of a face of $P$.
Since every $d$-polytope is 1-neigh\-bor\-ly, we assume $k \ge 2$ when we use the word ``$k$-neigh\-bor\-ly''.
For $k > \lfloor d/2\rfloor$, there is only one combinatorial type of a $k$-neigh\-bor\-ly $d$-polytope~--- a~$d$-simplex~\cite[p.~123]{Grunbaum:2003}.
Therefore, we suppose $d \ge 2k$.
A $\lfloor d/2\rfloor$-neigh\-bor\-ly polytope is called \emph{neigh\-bor\-ly}.

There exists a widespread feeling that neighborly polytopes are very common among convex polytopes~\cite{Henk:2004, Gillmann:2006, Grunbaum:2003}.
The current estimates for the number of combinatorial types of neighbborly $d$-polytopes are very close asymptotically to the ones of (all) $d$-polytopes~\cite{Padrol:2013}.
(In~\cite{Firsching:2017}, there are presented modern achievements in classification of neighborly polytopes with small dimension ($\le 10$) and number of vertices ($\le 12$).)
Moreover, $k$-neigh\-bor\-ly polytopes appear as faces (with superpolynomial number of vertices) of polytopes
associated with NP-hard combinatorial optimization problems~\cite{Maksimenko:2014, Maksimenko:2016}. 

The problem of estimating the number of facets $f_{d-1}(P)$  (where $P$ belongs to some class of $d$-polytopes) with respect to the number of vertices $f_0(P)$ is well known.
For the class of simplicial polytopes, the problem is known as the upper bound and the lower bound theorems~\cite[Chap.~10]{Grunbaum:2003}.
In~1970, P.~McMullen~\cite{McMullen:1970} stated that neighborly $d$-polytopes have the maximum number of facets over all $d$-po\-ly\-to\-pes with the same number of vertices.
In~\cite{Maksimenko:2010}, there was posed the conjecture that $f_{d-1}(P) \ge f_{0}(P)$ for a 2-neighborly $d$-polytope $P$.
The validity of the conjecture was proved for two cases: 1) $d \le 6$ and 2) $f_0(P) \le d+5$.

Let $\Delta_m(k) = \min\limits_P (f_{d-1}(P) - f_{0}(P))$, where  $f_{0}(P) = d + m$ and $P$ is a $k$-neighborly $d$-polytope.
Obviously, $\Delta_m(1) = 2 - m$ and it is attained on a prism over a $(m-1)$-simplex. 
Using the Gale diagrams~\cite[Sec. 6.3]{Grunbaum:2003}, it is easy to show that $\Delta_2(k) = (k+1)^2 - 2(k+1) = k^2 - 1$.
Below we prove the following
\begin{theorem}
	\label{th:d3}
	$\Delta_3(2) = 4$, $\Delta_3(3) = 15$, and~$\Delta_3(k) = 2 (k^2 - 1)$ for $k \ge 4$.
\end{theorem}
The proof of the theorem is based on the fact, that each $d$-polytope with $d+3$ vertices can be represented by a reduced Gale diagram~--- a set of $d+3$ points, placed at the center and at vertices of a regular $2n$-gon (see~\cite[Sec. 6.3]{Grunbaum:2003} and~\cite{Fusy:2006}).

%
%

\section{Reduced Gale diagrams}
\label{sec:GaleD}

A \emph{reduced Gale diagram} of a polytope $P$ consists of points in $\R^2$, placed at the center $O$ and the vertices 
of a regular $2n$-gon, $n \ge 2$.
For the sake of convinience, we enumerate the vertices of the $2n$-gon by numbers from $0$ to $2n-1$ 
going clockwise. (It does not matter what point will be the first.)
The $2n$ points have nonnegative integer labels (multiplicities) $m_0$, $m_1$, \dots, $m_{2n-1}$.
The center $O$ has label $m(O)$.
In~the following we assume
\[
m_i \stackrel{\text{def}}{=} m_{i\bmod 2n}.
\]
In~particular, $m_{i+2n} \stackrel{\text{def}}{=} m_{i-2n} \stackrel{\text{def}}{=} m_i$.
A pair of opposite labels $m_i$ and $m_{i+n}$, $i\in[n]$, is called a \emph{diameter}.
Hereinafter, $n$ is always the number of diameters of a reduced Gale diagram.
These labels have the following properties~(see~\cite[Sec.~6.3]{Grunbaum:2003} and~\cite{Fusy:2006}):
\begin{description}
	\item[P1] $m(O) + m_1 + \ldots + m_{2n-1} = d+3$. 
	(The sum of all labels equals the number of vertices of $P$.)
	\item[P2] $m_i + m_{i+n} > 0$ for every $i\in[n]$. 
	(The sum of the labels of a diameter is not equal to 0.
	Otherwise, we can construct this diagram on a $(2n-2)$-gon.)
	\item[P3] $m_{i-1} + m_i > 0$ for every $i\in[2n]$. 
	(Two neighbour vertices of the $2n$-gon cannot both have label 0.
	Otherwise the appropriate two diameters can be glued and we get a $(2n-2)$-gon.)					
	\item[P4] $\sum\limits_{j=i+1}^{i+n-1} m_{i} \ge 2$ for every $i\in[2n]$. 
	(The sum of the labels in any open half-plane is at least $2$.) 
\end{description}

A subset $S$ of these $d+3$ points\footnote{Take into account the multiplicities.} called a \emph{cofacet}
if the convex hull of $S$ is a simplex with $O$ in its relative interior
and $\ext(S) = \ext(\conv(S))$ (the set of vertices of $\conv(S)$ coincides with $S$).
Therefore, there are only three types of cofacets: 
the center $O$ itself, 
two opposite vertices of the $2n$-gon, 
three vertices of the $2n$-gon that form a triangle with $O$ in its interior.

\begin{theorem}{\cite[Sec.~6.3]{Grunbaum:2003}}
	\label{thm:GaleProperties}
	There are one-to-one correspondence between combinatorial types 
	of $d$-polytopes with $d+3$ vertices and reduced Gale diagrams with the properties P1--P4.
	Moreover,
	\begin{description}
		\item[F] The number of facets of a polytope is equal to the number of cofacets of the appropriate reduced Gale diagram.
		\item[S] An appropriate polytope is simplicial \textbf{iff} $m(O) = 0$ and $m_i m_{n+i} = 0$ for every $i\in [n]$. 
		\item[N] An appropriate polytope is $k$-neigh\-bor\-ly \textbf{iff} 
		$\sum\limits_{j=i+1}^{i+n-1} m_{i} \ge k+1$ for every $i\in[2n]$.
		(The sum of the labels in any open half-plane is at least $k+1$.)
		Such a diagram is called \emph{$k$-neigh\-bor\-ly}.
	\end{description}
\end{theorem}

Note that the properties P2 and P3 always may be obtained by making \emph{standard operations}:
\begin{description}
	\item[D] Delete the labels with $m_i = m_{n+i} = 0$ and reduce the number of diameters.
	\item[G] If $m_i = m_{i+1} = 0$, then glue $m_i$ and $m_{i+1}$, $m_{i+n}$ and $m_{i+1+n}$,
	and reduce the number of diameters.
\end{description}
The standard operations do not affect the $k$-neighborliness and the number of cofacets.
Thus, in the following we do not always care about checking the properties P2 and P3.

Note also that changing of the value of the label $m(O)$ does not affect the properties P2, P3, P4, and N, and does not change the difference $f_{d-1}(P) - f_0(P)$ for the appropriate polytope $P$.
Thus, without loss of generality, we assume $m(O) = 0$. 

We will say that a $k$-neigh\-bor\-ly reduced Gale diagram is \emph{extremal}
if the difference between the number of cofacets and the sum of labels is minimal
among all $k$-neigh\-bor\-ly reduced Gale diagrams
(in~other words, the difference is equal to $\Delta_{3}(k)$).

Turn here to examples.

\begin{example}
	\label{ex:1}
	Let $n=2$ and $m_i = k+1$, $i\in[4]$.
	Fig.~\ref{fig:ex1} shows such a diagram for $k=2$.
	The number of vertices of an appropriate $k$-neigh\-bor\-ly polytope is equal to $4(k+1)$,
	the number of facets equals $2(k+1)^2$, and the difference is $2(k^2-1)$.
\end{example}

\begin{example}
	\label{ex:2}
	Let $n = k+2$ and all the labels are equal to 1 (see fig. \ref{fig:ex2}).
	Hence, the number of vertices equals $2k+4$, the number of facets equals $2 \binom{k+2}{3} + k+2$,
	and the difference is $(k+2)(k^2+k-3)/3$.
	Note that the difference is less than in the example \ref{ex:1} for $k \le 3$.
\end{example}

\begin{example}
	\label{ex:3}
	Let $n = 2k+3$, $m_{2i} = 0$, and $m_{2i-1} = 1$ for $i \in [n]$.
	Fig. \ref{fig:ex3} shows the case $k=2$.
	An appropriate $k$-neigh\-bor\-ly polytope is a simplicial one.
	It has dimension $2k$,
	$2k+3$ vertices and $(2k+3) (k+2) (k+1) / 6$ facets.
	The difference between facets and vertices is $(2k+3) (k+4) (k-1) / 6$
	and it is greater than in the previous examples for every $k\ge 2$.
	Hence, this diagram is not extremal.
\end{example}

\begin{figure}
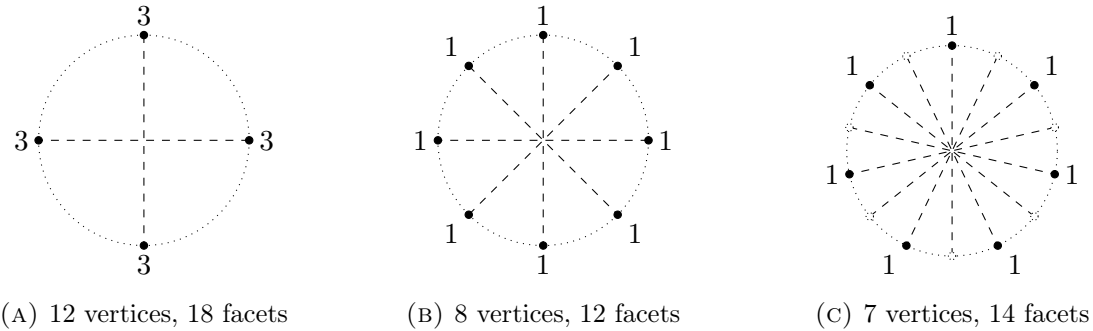

	\begin{minipage}[t]{.33\linewidth}
		\centering\picRomb
		\subcaption{12 vertices, 18 facets}\label{fig:ex1}
	\end{minipage}%
	\begin{minipage}[t]{.33\linewidth}
		\centering\picOnes
		\subcaption{8 vertices, 12 facets}\label{fig:ex2}
	\end{minipage}
	\begin{minipage}[t]{.33\linewidth}
		\centering\picCyclic
		\subcaption{7 vertices, 14 facets}\label{fig:ex3}
	\end{minipage}
	\caption{Reduced Gale diagrams for some 2-neighborly polytopes}\label{fig:2neighborly}
\end{figure}


%
%

\section{Minimal Gale diagrams}

A $k$-neigh\-bor\-ly reduced Gale diagram is called \emph{minimal}
if reducing any of the labels $m_i$ violates the condition N.
It is easy to prove that an extremal diagram is a minimal one (since reducing of a label implies reducing of the number of cofacets).
Thus, below we consider only minimal diagrams.

Let us list two obvious properties of a minimal $k$-neigh\-bor\-ly reduced Gale diagram:
\begin{align*}
m_i &\le k+1, \quad i\in[2n],\\ 
n   &\le 2k+3. 
\end{align*}
In~the case $n = 2k+3$, there is only one minimal $k$-neigh\-bor\-ly reduced Gale diagram (see example \ref{ex:3}).
However, the diagram is not extremal. Thus, we are only interested in the cases 
\begin{equation}
n  \le 2k+2. \label{eq:UpN}
\end{equation}
On the other hand, $n \ge 2$.
In~the case $n = 2$, the only minimal $k$-neigh\-bor\-ly reduced Gale diagram has labels $m_i = k+1$, $i\in[4]$ (see example \ref{ex:1}).

\begin{theorem}[\cite{Marcus:1981}]
	\label{th:Marcus}
	The sum of labels of any minimal $k$-neigh\-bor\-ly reduced Gale diagram is not greater than $4(k+1)$.
	If the sum equals $4(k+1)$, then such Gale diagram is a $4$-gon with labels $m_1 = m_2 = m_3 = m_4 = k+1$.
\end{theorem}

Now we are ready to start the proof of Theorem~\ref{th:d3}.
It consists of the following steps.
First of all, we show that $m_i + m_{i-1} \ge 2$, $i \in [2n]$, for an extremal Gale diagram
(Lemmas \ref{lem:ShiftZero}--\ref{lem:NonZero}).
It gives us the possibility for enumerating all extremal Gale diagrams for $k \le 6$ with the help of a computer (Proposition~\ref{pr:k6}).
The case $k \ge 6$ is analyzed in the section~\ref{sec:kg6}.

%
%

\section{Local properties of extremal Gale diagrams}


\begin{lemma} 
	\label{lem:ShiftZero}
	Let $m_i$ be a positive label 
	in a $k$-neigh\-bor\-ly reduced Gale diagram with $n$ diameters.
	In~addition, let $m_{i+n} = 0$ and $m_{i+n-1} > 0$ (see fig.~\ref{fig:shift1}).
	Then the \emph{displacement operation}
	$$
	m_i := m_i - 1 \quad \text{and} \quad m_{i-1} := m_{i-1} + 1
	$$
	reduces the total number of facets by at least $k$.
	Moreover, the new Gale diagram will be $k$-neigh\-bor\-ly
	whenever $\sum_{j = i}^{i+n-2} m_j \ge k+2$.
\end{lemma}
\begin{figure}
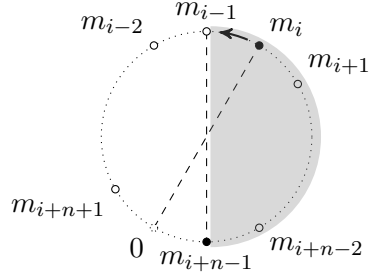

	\centering
	\picShift
	\caption{The displacement operation. 
		The sum of lables in the gray semicircle must be greater than $k+2$ 
		for preserving $k$-neighborliness.}\label{fig:shift1}
\end{figure}
\begin{proof}
	After such displacing, there will be lost
	\[
	m_{i+n-1} \sum_{j = i+n+1}^{i-2} m_j
	\]
	cofacets of the form $\{i, i+n-1, j\}$, where $j \in [i+n+1, i-2]$.
	At the same time, there will appear $m_{i+n-1}$ new cofacets of the form $\{i-1, i+n-1\}$.
	Note that 
	\[
	m_{i+n-1} > 0 \quad \text{and} \quad \sum_{j = i+n+1}^{i-2} m_j = \sum_{j = i+n}^{i-2} m_j \ge k+1.
	\]
	Therefore, after this operation, the total number of cofacets will be reduced at least~by
	\[
	m_{i+n-1} (k+1 - 1) \ge k.
	\]
	
	The displacement operation reduces the sum $\sum_{j = i}^{i+n-2} m_j$ (the gray semicircle in the~fig.~\ref{fig:shift1}) by 1.
	The sums in other half-planes are not reduced.
	If the inequality $\sum_{j = i}^{i+n-2} m_j \ge k+2$ is fulfilled for the source diagram,
	then the displaced diagram is $k$-neigh\-bor\-ly.
\end{proof}

A reduced Gale diagram is called \emph{simplicial}
if the appropriate polytope is a simplicial one.

\begin{lemma} 
	\label{lem:Simplicial}
	Simplicial Gale diagram is not extremal.
\end{lemma}
\begin{proof}
	As it was remarked in \eqref{eq:UpN}, 
	an extremal $k$-neigh\-bor\-ly Gale diagram has no more than $2k+2$ diameters.
	By property~S from Theorem~\ref{thm:GaleProperties}, each diameter of a simplicial reduced Gale diagram has exactly
	one positive label.
	Without loss of generality, we suppose that $m_{2i} = 0$ and $m_{2i-1} > 0$ for every $i\in[n]$
	(remember the property P3).
	Hence, $n$ is odd. That is 
	$$
	n \in [3, 2k+1].
	$$
	Thus, at least one label of this diagram is greater than~1.
	Let us suppose, that $m_t > 1$ for some odd $t \in [2n]$.
	More precisely, there are at least three such labels,
	since $\sum_{i = t+1}^{t+n-1} m_i \ge k+1$ and $\sum_{i = t-n+1}^{t-1} m_i \ge k+1$.
	
	\begin{figure}
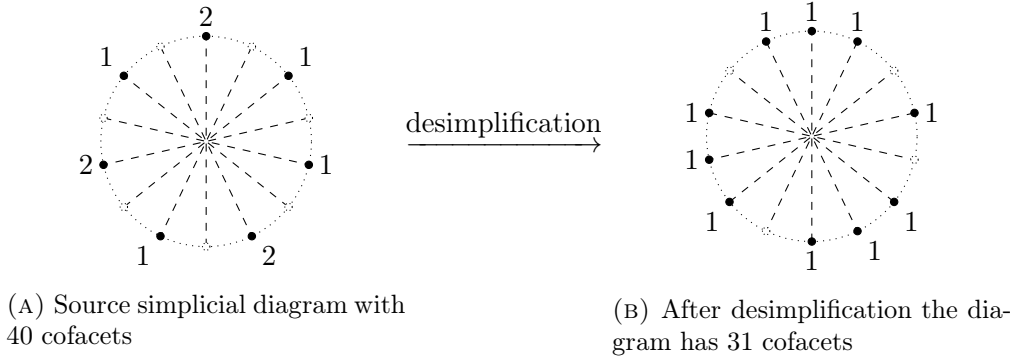

		\centering
		\begin{minipage}[c]{.33\linewidth}
			\centering\picSimplicialA
			\subcaption{Source simplicial diagram with 40 cofacets}\label{fig:desim1}
		\end{minipage}%
		\raisebox{1.5\baselineskip}{$\underrightarrow{\text{desimplification}}$}
		\begin{minipage}[c]{.33\linewidth}
			\centering\picSimplicialB
			\subcaption{After desimplification the~diagram has 31 cofacets}\label{fig:desim2}
		\end{minipage}
		\caption{Desimplification}\label{fig:Desimplification}
	\end{figure}
	
	Let us consider the following procedure.
	We call it \emph{desimplification} (see fig.~\ref{fig:Desimplification}).
	All positive labels are reduced by~1, 
	zero labels (except $m(O)$) are incremented by~1.
	It is clear, that the new diagram will keep $k$-neighborliness
	and will become nonsimplicial (since at least one label was greater than~1).
	
	It remains to show that the number of cofacets is reduced after desimplification.
	To do this we divide all $v$ points on \emph{movable} and \emph{static}, where $v = d+3$.
	There are exactly $n$ movable and $v-n$ static points.
	For the sake of convenience, we suppose that every movable point
	is moved from its position $2i-1$ to the position $2i$.
	All static points remain in their positions with odd indices.
	
	By property S, every cofacet of the source diagram consists of 3 points (see fig.~\ref{fig:MMS}).
	The number of cofacets consisted only of static points is not changed after desimplification.
	The same is true for the cofacets consisted of movable points.
	It remains to estimate the change of the number of cofacets of types $\{m, m, s\}$, $\{m, s, s\}$, and $\{m, s\}$, where $m$ is a movable and $s$ is a static point.
	\begin{figure}
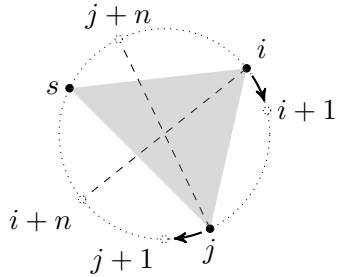

		\centering
		\picMMS
		\caption{A cofacet $\{i, j, s\}$.}\label{fig:MMS}
	\end{figure}
	
	Let us consider a cofacet of the type $\{m, m, s\}$.
	More precisely, we consider cofacets $\{i, j, s\}$, 
	where $i$, $j$, $s$ are odd (see fig.~\ref{fig:MMS}).
	Without loss of generality, we suppose that 
	\[
	i < j < i + n \quad\text{and}\quad i+n < s < j+n.
	\]
	After desimplification, $i$ and $j$ are moved to $i+1$ and $j+1$.
	They may form cofacets of the type $\{m, m, s\}$ 
	with static points $s'\in[i+n+2, j+n]$.
	But $i+n+2$ and $j+n$ are even and have no static points.
	Thus $s' \in [i+n+3, j+n-1]$ and we lose cofacet(s) $\{i, j, i+n+1\}$
	(if there are static points in the position $i+n+1$).
	Therefore, for every static point~$s$, there are lost $(n-1)/2$ cofacets
	with $i = s-n-1$ and $j \in [i+2, i+n-1]$.
	Consequently, the total number of $\{m, m, s\}$-type is reduced by $(v-n)(n-1)/2$.
	
	Continuing in the same way, it is easy to see that the number of $\{m, s, s\}$-type cofacets
	is reduced by 1 for every pair of distinguished static points.
	Recall that we have at least 3 such points.
	
	The number of $\{m, s\}$-type cofasets is equal to 0 for the source diagram
	and it is equal to the number $v-n$ of static points after desimplification.
	
	Therefore, the total number of cofacets is reduced by at least $3 + (v-n)(n-3)/2$
	after desimplification.
\end{proof}

A diameter $\{m_i, m_{i+n}\}$ is called \emph{complete} if $m_i > 0$ and $m_{i+n} > 0$.
Otherwise, $\{m_i, m_{i+n}\}$ is called \emph{incomplete}.
From the just proved Lemma~\ref{lem:Simplicial}, 
it follows that an extremal reduced Gale diagram has at least one complete diameter.

Two diameters $\{m_i, m_{i+n}\}$ and $\{m_{i+1}, m_{i+1+n}\}$ are called \emph{consecutive}.

\begin{lemma} 
	\label{lem:Dpairs}
	An extremal $k$-neigh\-bor\-ly Gale diagram has no consecutive incomplete diameters.
\end{lemma}
\begin{figure}
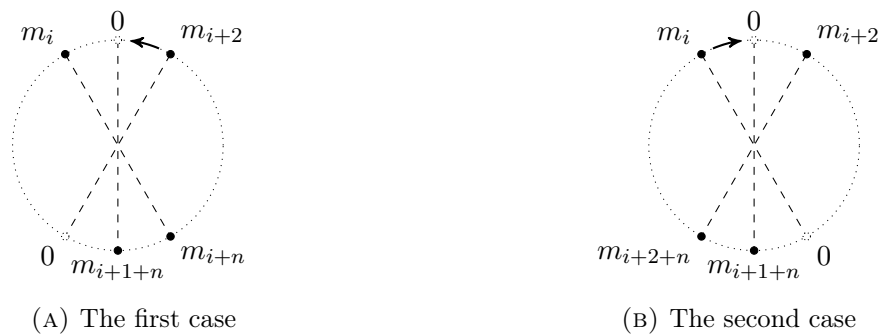

	\begin{minipage}[t]{.5\linewidth}
		\centering\picConsecutiveA
		\subcaption{The first case}\label{fig:Consecutive1}
	\end{minipage}%
	\begin{minipage}[t]{.5\linewidth}
		\centering\picConsecutiveB
		\subcaption{The second case}\label{fig:Consecutive2}
	\end{minipage}
	\caption{Two consecutive incomplete diameters}\label{fig:Consecutive}
\end{figure}
\begin{proof}
	Suppose to the contrary that some extremal $k$-neigh\-bor\-ly Gale diagram $D$ has consecutive incomplete diameters.
	From Lemma~\ref{lem:Simplicial},
	we know that $D$ has at least one complete diameter.
	Hence, there are three diameters $\{m_i, m_{i+n}\}$, $\{m_{i+1}, m_{i+1+n}\}$, $\{m_{i+2}, m_{i+2+n}\}$
	such that one of the two conditions is satisfied (see fig.~\ref{fig:Consecutive}):
	\begin{enumerate}
		\item The first diameter is complete and the last two are incomplete.
		\item The last diameter is complete and the first two are incomplete.
	\end{enumerate}
	
	We examine only the first case. The second case is examined by analogy.
	
	By the condition N from Theorem~\ref{thm:GaleProperties}, 
	$$
	\sum_{j=i+1}^{i+n-1} m_j \ge k+1.
	$$
	Since $m_{i+1} = 0$ and $m_{i+n} > 0$, we get 
	$$
	\sum_{j=i+2}^{i+n} m_j \ge k+2.
	$$
	From Lemma~\ref{lem:ShiftZero} (see also fig.~\ref{fig:shift1}),
	it follows that the displacement operation
	$$
	m_{i+2} := m_{i+2} - 1 \quad \text{and} \quad m_{i+1} := m_{i+1} + 1
	$$
	reduces the total number of facets by at least $k$.
	Moreover, the new Gale diagram will be $k$-neigh\-bor\-ly.
	Therefore, the source diagram $D$ is not extremal.
\end{proof}


\begin{lemma} 
	\label{lem:NonZero}
	For an extremal $k$-neigh\-bor\-ly Gale diagram, $m_i + m_{i-1} \ge 2$, $i \in [2n]$.
\end{lemma}
\begin{figure}
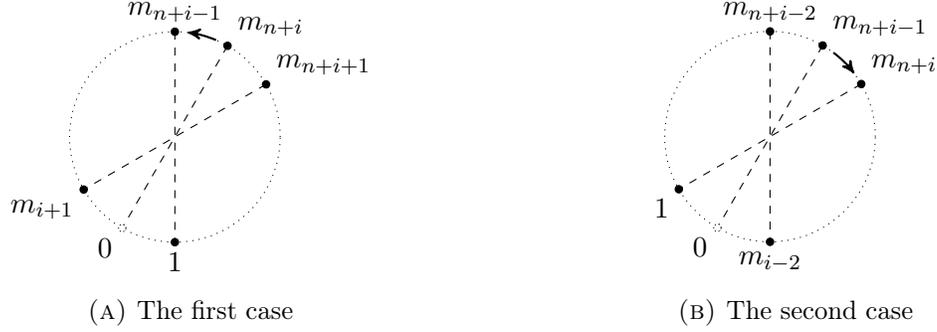

	\begin{minipage}[t]{.5\linewidth}
		\centering\picSumTwoA
		\subcaption{The first case}\label{fig:SumTwo1}
	\end{minipage}%
	\begin{minipage}[t]{.5\linewidth}
		\centering\picSumTwoB
		\subcaption{The second case}\label{fig:SumTwo2}
	\end{minipage}
	\caption{The sum of two consecutive labels cannt be less than 2.}\label{fig:SumTwo}
\end{figure}
\begin{proof}
	Suppose to the contrary that there is $i\in[2n]$ such that $m_i + m_{i-1} = 1$.
	(By the property P3, $m_i + m_{i-1} > 0$.)
	Hence one of two diameters $\{m_i, m_{n+i}\}$ and $\{m_{i-1}, m_{n+i-1}\}$ is incomplete.
	By Lemma~\ref{lem:Dpairs}, there are two symmetrical cases (see fig.~\ref{fig:SumTwo}):
	\begin{enumerate}
		\item $m_{i-1} = 1$, $m_{i} = 0$, $m_{i+1} > 0$, $m_{n+i-1} > 0$, $m_{n+i} > 0$, $m_{n+i+1} > 0$.
		\item $m_{i-1} = 0$, $m_{i} = 1$, $m_{i-2} > 0$, $m_{n+i-2} > 0$, $m_{n+i-1} > 0$, $m_{n+i} > 0$.
	\end{enumerate}
	
	We examine only the first case.
	
	By the condition N, 
	$$
	\sum_{j=n+i+2}^{2n+i} m_j \ge k+1.
	$$
	Since $m_{n+i} + m_{n+i+1} > m_{2n+i-1} + m_{2n+i}$, we get 
	$$
	\sum_{j=n+i}^{2n+i-2} m_j \ge k+2.
	$$
	From Lemma~\ref{lem:ShiftZero},
	it follows that the displacement operation
	$$
	m_{n+i} := m_{n+i} - 1 \quad \text{and} \quad m_{n+i-1} := m_{n+i-1} + 1
	$$
	reduces the total number of facets by at least $k$.
	Moreover, the new Gale diagram will be $k$-neigh\-bor\-ly.
	Therefore, the source diagram $D$ is not extremal.
\end{proof}

\begin{lemma}
	\label{pr:MaxL}
	Consider an extremal $k$-neigh\-bor\-ly Gale diagram.
	Let $i \in [2n]$ and $q \ge 0$. 
	If
	\begin{equation}
	\label{eq:UpBoundM}
	\sum_{j = i+1-t}^{i+n-1-t} m_j \ge m_i + q \quad \text{for every $t\in[n-1]$}
	\end{equation}
	(in~other words, in any open half-plane containing $m_i$, 
	the sum of labels is not less than $m_i + q$),
	then $m_i \le k+1 - q$.
\end{lemma}
\begin{proof}
	If $m_i > k+1 - q$, then in any open half-plane containing $m_i$ the sum of labels is greater than $k+1$.
	Thus, the reducing $m_i := m_i - 1$ does not violate the property~N
	and the diagram is not minimal.
\end{proof}

\begin{lemma}
	\label{pr:UpB}
	For an extremal $k$-neigh\-bor\-ly Gale diagram with $n \ge 3$ diameters, 
	\begin{equation}
	\label{eq:MaxMi}
	m_i \le \begin{cases}
	k+5 - n & \text{if $n$ is even}, \\
	k+4 - n & \text{if $n$ is odd},
	\end{cases}
	\quad
	\text{and}
	\quad
	n \le \begin{cases}
	k+2 & \text{if $k$ is even}, \\
	k+3 & \text{if $k$ is odd}.
	\end{cases}
	\end{equation}
\end{lemma}

\begin{proof}
	From Lemma~\ref{lem:NonZero}, it follows that the inequality \eqref{eq:UpBoundM}
	holds for every $i \in [2n]$ and
	$$
	q = \begin{cases}
	n-4 & \text{if $n$ is even}, \\
	n-3 & \text{if $n$ is odd}.
	\end{cases}
	$$
	By Lemma~\ref{pr:MaxL}, we get 
	$m_i \le k+5 - n$ for even $n$ and $m_i \le k+4 - n$ for odd $n$.
	
	This implies that $n \le k+3$ for odd $k$.
	
	Now suppose that $k$ is even. Hence $n \le k+4$.
	Note also that $m_i \le 1$ for every $i\in[2n]$ and $n\in\{k+3, k+4\}$.
	By Lemma~\ref{lem:NonZero}, we get $m_i = 1$ in this case.
	Thus, by removing 1 or 2 diameters, we can transform such a diagram to one in the example~\ref{ex:2}.
	Therefore, $n \le k+2$ for an extremal Gale diagram if $k$ is even.
\end{proof}

\begin{proposition}
	\label{pr:k6}
	$\Delta_3(2) = 4$, $\Delta_3(3) = 15$, 
	and the equalities are attained in the example~\ref{ex:2}.
	For $k\in \{4, 5, 6\}$, $\Delta_3(k) = 2 (k^2 - 1)$
	and the best difference is attained in the example~\ref{ex:1}.
\end{proposition}

\begin{proof}
	The property N, Theorem~\ref{th:Marcus}, Lemmas~\ref{lem:NonZero} and~\ref{pr:UpB} impose strong restrictions.
	For small values of $k$, this allows us to look over all such diagrams and find $\Delta_3(k)$ by using a computer (see the CPP-code in Appendix~\ref{app:gale.cpp}).
	For $k \le 6$ it takes less than 1 minute.
\end{proof}



%
%

\section{\texorpdfstring{The case $k \ge 6$}{The case k > 5}}
\label{sec:kg6}

The goal of this section is to prove $\Delta_3(k) = 2 (k^2 - 1)$ for $k \ge 6$.
From Theorem~\ref{th:Marcus}, we know that the sum of labels 
of an extremal $k$-neigh\-bor\-ly Gale diagram $D$ is not greater than $4(k+1)$.
Thus, it is sufficient to show that the number of cofacets $f=f(D)$
cannot be less than 
$$
2(k^2-1) + 4(k+1) = 2(k+1)^2 \quad \text{for $k \ge 6$.}
$$
(On the other hand, this value is attained in the example~\ref{ex:1}.)

We partition the task into three cases:
\begin{description}
	\item[Lemma~\ref{lem:Ge1}] There is an open half-plane with only one positive label in the Gale diagram. See fig.~\ref{fig:lem1}.
	\item[Lemma~\ref{lem:Ge2}] There is an open half-plane with exactly two positive labels
	and one of them has value 1. See fig.~\ref{fig:lem2}.
	\item[Lemma~\ref{lem:Induction}] In~every open half-plane of the considered Gale diagram, there are at least two positive labels and if there are exactly two positive labels, then their values are not less than 2.
	This condition allows us to use Lemma~\ref{pr:MaxL} with $q=2$.
\end{description}

\begin{lemma} 
	\label{lem:Ge1}
	Let $D$ be an extremal $k$-neigh\-bor\-ly Gale diagram.
	If there is an open half-plane with only one positive label,
	then $f=f(D) \ge 2(k+1)^2$ for $k \ge 2$.
\end{lemma}
\begin{figure}
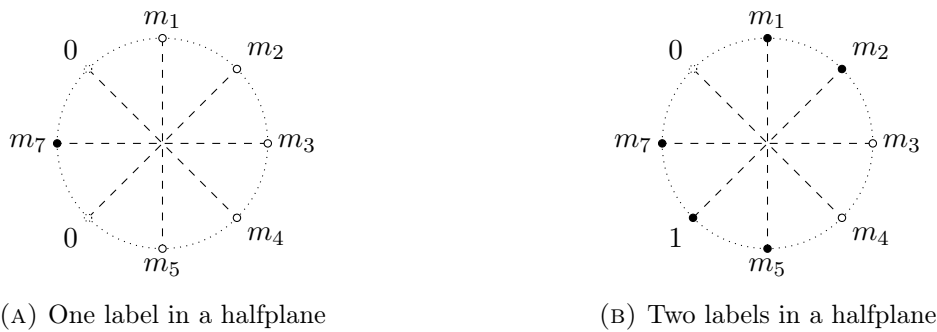

	\begin{minipage}[t]{.5\linewidth}
		\centering\picHalfA
		\subcaption{One label in a halfplane}\label{fig:lem1}
	\end{minipage}%
	\begin{minipage}[t]{.5\linewidth}
		\centering\picHalfB
		\subcaption{Two labels in a halfplane}\label{fig:lem2}
	\end{minipage}
	\caption{Two simple cases}\label{fig:2lemmas}
\end{figure}
\begin{proof}
	The general case is shown in the fig.~\ref{fig:lem1}.
	By assumption, $m_6 = m_8 = 0$.
	For convenience, we use the notation
	$$
	p = k+1.
	$$
	First note that $m_7 = p$, since the diagram is extremal and $k$-neigh\-bor\-ly.
	Some of the other labels may equal zero and the number of diameters may be less than 4 (and cann't be greater than 4, due to the property P3).
	By the property N,
	\begin{equation}
	\label{eq:LGe1}
	m_1 + m_2 \ge p, \quad m_4 + m_5 \ge p, \quad m_2 + m_3 + m_4 \ge p.
	\end{equation}
	Let us count the number of cofacets:
	$$
	f = m_1 m_5 + m_3 m_7 + m_1 m_4 m_7 + m_2 m_4 m_7 + m_2 m_5 m_7.
	$$
	Note that
	$$
	f \ge m_1 m_4 m_7 + m_2 m_4 m_7 \ge m_4 p^2 \quad \text{and} \quad
	f \ge m_2 m_4 m_7 + m_2 m_5 m_7 \ge m_2 p^2.
	$$
	Consequently, if $m_2 \ge 2$ or $m_4 \ge 2$, then $f \ge 2 p^2 = 2 (k+1)^2$.
	It remains to analyze four cases:
	\begin{enumerate}
		\item If $m_2 = m_4 = 0$, then, according to \eqref{eq:LGe1}, we have
		$m_1 \ge p$, $m_5 \ge p$, $m_3 \ge p$. Hence $f = m_1 m_5 + m_3 m_7 \ge 2 p^2$.
		\item If $m_2 = 0$ and $m_4 = 1$, then $m_1 \ge p$, $m_5 \ge p-1$, $m_3 \ge p-1$.
		Hence $f \ge p(p-1) + (p-1)p + p^2 = 3 p^2 - 2p \ge 2 p^2$ for $p\ge 2$.
		\item The case $m_2 = 1$ and $m_4 = 0$ is analysed by analogy with the previous one.
		\item If $m_2 = m_4 = 1$, then $m_1 \ge p-1$, $m_5 \ge p-1$, $m_3 \ge p-2$.
		Therefore $f \ge (p-1)^2 + (p-2)p + (p-1)p + p + (p-1)p = 4 p^2 - 5p + 1 > 2 p^2$ for $p\ge 3$. \qed
	\end{enumerate}
	\renewcommand{\qed}{}\end{proof}

\begin{lemma} 
	\label{lem:Ge2}
	Let $D$ be an extremal $k$-neigh\-bor\-ly Gale diagram.
	If there is an open half-plane with exactly two positive labels
	and the value of one of them is equal to~1,
	then $f=f(D) \ge 2(k+1)^2$ for $k \ge 4$.
\end{lemma}

\begin{proof}
	By Lemma~\ref{lem:NonZero}, the sum of two consecutive labels cannot be less than 2.
	Consequently, the neighbors of the label with value 1 must have positive values.
	Therefore, the number of diameters of the diagram cann't be greater than 4.
	The general case is shown in the fig.~\ref{fig:lem2}.
	By assumption, $m_8 = 0$.
	
	Here, we do not need to consider the cases analyzed in Lemma~\ref{lem:Ge1}.
	Hence 
	\begin{equation}
	\label{eq:m12}
	m_1 > 0, \quad m_2 > 0
	\end{equation}
	(the corresponding points are painted in fig.~\ref{fig:lem2}).
	By Lemma~\ref{lem:NonZero}, 
	\begin{equation}
	\label{eq:m5}
	m_5 > 0.
	\end{equation}
	
	Let $p = k+1$. By the property N,
	\begin{align}
	m_7 + 1 &\ge p, \label{eq:G1} \\
	m_1 + m_2 &\ge p, \label{eq:G2} \\
	m_2 + m_3 + m_4 &\ge p,  \label{eq:G3} \\
	m_3 + m_4 + m_5 &\ge p,  \label{eq:G4} \\
	m_4 + m_5 + 1 &\ge p. \label{eq:G5}
	\end{align}
	Let us count the number of cofacets:
	$$
	f = m_2 + m_3 m_7 + m_2 m_5 m_7
	+ m_4 m_7 (m_1 + m_2) + m_1 (m_3 + m_4 + m_5).
	$$
	In particular, by using \eqref{eq:G1} and \eqref{eq:G5} we obtain
	\begin{equation}
	\label{eq:Lem2}
	\begin{aligned}
	f &\ge m_3 m_7 + m_2 m_7 (m_4 + m_5) + m_1 (m_3 + m_4 + m_5) \\
	&\ge m_3 (p-1) + m_2 (p-1)^2 + m_1 (m_3 + m_4 + m_5).
	\end{aligned}
	\end{equation}
	
	Further proof is reduced to an analysis of all possible cases regarding the values of labels $m_4$ and $m_2$.
	In~all cases, we get $f \ge 2 p^2$ for $p \ge 5$.
	\begin{align*}
	\intertext{\indent $\mathbf{m_4 \ge 2}$.
		By the inequalities \eqref{eq:m12}, \eqref{eq:m5}, \eqref{eq:G1}, \eqref{eq:G4}, we have
	}
	f &\ge m_2 (1 + m_5 m_7) + m_4 m_7 (m_1 + m_2) + m_1 (m_3 + m_4 + m_5) \\
	&\ge m_2 p + 2 (p-1) p + m_1 p \ge 2 p^2.\\
	\intertext{\indent $\mathbf{m_2 \ge 3}$.
		By using \eqref{eq:Lem2}, \eqref{eq:m12}, \eqref{eq:G4}, we obtain
	}
	f &\ge 3 (p-1) (p-1) + p \ge 2 p^2.\\
	\intertext{\indent $\mathbf{m_2 = 2}$ \textbf{and} $\mathbf{m_4 \le 1}$.
		Hence $m_1 \ge p-2$, $m_3 \ge p-3$, $m_5 \ge p-2$, and
	}
	f &\ge (p-3)(p-1) + 2(p-1)^2 + (p-2)(p-3 + p-1) \ge 2 p^2.\\
	\intertext{\indent $\mathbf{m_2 = 1}$ \textbf{and} $\mathbf{m_4 \le 1}$.
		This implies that $m_1 \ge p-1$, $m_3 \ge p-2$, $m_5 \ge p-2$, 
	}
	f &\ge (p-2)(p-1) + (p-1)^2 + (p-1)(p-2+p-1) \ge 2p^2. \hfil \qed
	\end{align*}
	\renewcommand{\qed}{}\end{proof}

\begin{lemma} 
	\label{lem:Induction}
	Let $D$ be an extremal $k$-neigh\-bor\-ly Gale diagram.
	Suppose that in every open half-plane, there is at least two positive labels,
	and if there are exactly two positive labels, then their values are not less than 2.
	Then $f(D) \ge 2(k+1)^2$ for $k \ge 5$.
\end{lemma}

\begin{proof}
	The proof is by induction over $k$.
	The case $k=5$ is covered by Proposition~\ref{pr:k6}.
	
	Suppose that the statement of lemma is true for $k = t$, $t \ge 5$.
	
	Let us consider an extremal $(t+1)$-neigh\-bor\-ly Gale diagram $D$ with labels $\{m_0, \dots, m_{2n-1}\}$.
	By Lemma~\ref{lem:Simplicial}, 
	the diagram has at least one complete diameter $\{m_{j^*}, m_{n+j^*}\}$.
	Reduce $m_{j^*}$ and $m_{n+j^*}$ by 1.
	Obviously, the new diagram $D'$ is at least $t$-neigh\-bor\-ly.
	For the labels of $D'$ we will use the following notations (see fig.~\ref{fig:Induction}):
	\[
	\begin{aligned}
	x_1 &= m_{j^*} - 1,   & y_1 &= m_{n+j^*} - 1,\\
	x_i &= m_{j^*+i-1}, & y_i &= m_{n+j^*+i-1}, \ i\in[2,n].
	\end{aligned}
	\] 
	\begin{figure}
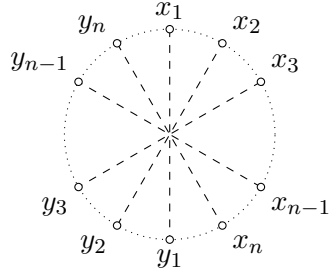

		\centering\picInduction
		\caption{The induction step}\label{fig:Induction}
	\end{figure}
	
	Observe that
	\begin{equation}
	\label{eq:XiYi}
	\sum_{i=2}^n x_i \ge t+2 \quad \text{and} \quad \sum_{i=2}^n y_i \ge t+2.
	\end{equation}
	
	By the induction hypothesis, $f(D') \ge 2(t+1)^2$.
	Hence it is sufficient to prove that
	\[
	f(D) - f(D') \ge 4(t+1) + 2.
	\]
	
	By assumption, the source diagram $D$ satisfies 
	the conditions of Lemma~\ref{pr:MaxL} with $q=2$.
	Thus
	\begin{equation*}
	m_i \le t+2 - 2,  \quad \forall i\in[2n],
	\end{equation*}
	and
	\begin{equation}
	\label{eq:leql}
	x_j \le t, \quad y_j \le t,  \quad \forall j\in[n].
	\end{equation}
	
	Let us count the difference between $f(D)$ and $f(D')$.
	We start from the cofacets with exactly two points.
	The diagram $D$ has exactly $x_1 + y_1 + 1$ such cofacets that do not present in $D'$.
	When reducing $m_{j^*}$ and $m_{n+j^*}$, there are also lost triangular cofacets of the following types:
	\begin{enumerate}
		\item $\{j^*, r, s\}$, \phantom{${} + n$\!}where $r\in[j^*+2, j^* + n - 1]$ and $s\in[j^*+n+1, r+n-1]$.
		\item $\{j^*+n, r, s\}$, where $r\in[j^*+1,\, j^* + n - \,2]$ and $s\in[r+n+1, j^*+2n-1]$.
	\end{enumerate}
	In~total, we get
	\begin{equation}
	\label{eq:D1D2}
	f(D) - f(D') = \left(\sum_{i=2}^n x_i \right) \left(\sum_{i=2}^n y_i \right) - \sum_{i=2}^n x_i y_i + x_1 + y_1 + 1.
	\end{equation}
	
	Let us suppose that $\min\{x_i, y_i\} \le 2$ for every $i \in[2, n]$.
	Hence
	\[
	\sum_{i=2}^n x_i y_i \le 2 \left(\sum_{i=2}^n y_i \right) 
	+ 2 \left(\sum_{i=2}^n x_i \right) - 4.
	\] 
	Combining this with \eqref{eq:XiYi}, we obtain
	\begin{equation*}
	\begin{aligned}
	f(D) - f(D') &\ge \left(\sum_{i=2}^n x_i  - 2\right) \left(\sum_{i=2}^n y_i  - 2\right) + x_1 + y_1 + 1 \\
	&\ge t^2 + 1 \ge 4(t+1) + 2 \quad \text{for } t \ge 5.
	\end{aligned}
	\end{equation*}
	
	Now we suppose that
	$$
	\max_{i \in[2, n]}(\min\{x_i, y_i\}) \ge 3.
	$$
	Hence $x_{i'} \ge 3$ and $y_{i'} \ge 3$ for some $i' \in [2, n]$.
	In~addition, $x_{i'} \le t$ and $y_{i'} \le t$ by~\eqref{eq:leql}.
	Let $I = [n]\setminus \{1, i'\}$ and 
	\[
	G(x_{i'}, y_{i'}) = \left(\sum_{i\in I} x_i \right) \left(\sum_{i\in I} y_i \right) - \sum_{i\in I} x_i y_i.
	\]
	It is clear that $G(x_{i'}, y_{i'}) \ge 0$.
	Combining this with \eqref{eq:D1D2}, we get
	\begin{equation*}
	\begin{aligned}
	f(D) - f(D') &= G(x_{i'}, y_{i'}) + y_{i'} \sum_{i\in I} x_i + x_{i'} \sum_{i\in I} y_i + x_1 + y_1 + 1 \\
	&\ge y_{i'} \sum_{i\in I} x_i + x_{i'} \sum_{i\in I} y_i + x_1 + y_1 + 1.
	\end{aligned}
	\end{equation*}
	Since the diagram $D'$ is $t$-neigh\-bor\-ly, we have
	\[
	\sum_{i=1}^n (x_i+y_i) - x_{i'} - y_{i'} \ge 2(t+1).
	\]
	Thus
	\[
	x_1 + y_1 \ge 2t + 2 - \sum_{i \in I} (x_i+y_i)
	\]
	and
	\[
	f(D) - f(D') \ge (y_{i'} - 1) \sum_{i\in I} x_i + (x_{i'} - 1) \sum_{i\in I} y_i + 2t + 3.
	\]
	
	From \eqref{eq:XiYi}, we have
	\[
	\sum_{i\in I} x_i \ge t+2 - x_{i'} \quad \text{and} \quad \sum_{i\in I} y_i \ge t+2 - y_{i'}.
	\]
	Hence
	\[
	f(D) - f(D') \ge (y_{i'} - 1) (t+2 - x_{i'}) + (x_{i'} - 1) (t+2 - y_{i'}) + 2t + 3.
	\]
	For the sake of convenience, we use the notations:
	\[
	p = t+1, \qquad x = x_{i'} - 1, \qquad y = y_{i'} - 1.
	\]
	Hence $p \ge 6$ (by assumption), $x \in [2, p-2]$, $y \in [2, p-2]$, and
	\[
	f(D) - f(D') \ge y (p - x) + x (p - y) + 2p + 1.
	\]
	But
	\[
	y (p - x) + x (p - y) \ge 4 (p - 2) \qquad \text{for } x \in [2, p-2],\  y \in [2, p-2].
	\]
	Therefore,
	\[
	f(D) - f(D') \ge 6p - 7 \ge 4p + 2 = 4(t+1) + 2 \quad \text{if } t \ge 5.
	\hfil \qed
	\]
	\renewcommand{\qed}{}\end{proof}

\section*{Acknowledgements}

The author wants to thank Arnau Padrol for his useful comments.

%
%

\appendix

\lstset{
	breaklines=true,
	tabsize=4, 
	formfeed=\newpage, 
	extendedchars=true, 
	basicstyle=\ttfamily, 
	commentstyle=\rmfamily\itshape, 
	stringstyle=\slshape, 
	numbers=left, 
	numbersep=1em, 
	stepnumber=1, 
	numberstyle=\footnotesize\color{black}, 
}

\section{The source code for the enumerating\\ extremal k-neighborly Gale diagrams}
\label{app:gale.cpp}

\lstdefinestyle{cpp}{
	language=[ANSI]C++,
	morekeywords={string, list} 
}

\begin{lstlisting}[style=cpp]
// This program finds a k-neighborly Gale diagram 
// with the minimal difference between cofacets and vertices (sum of labels)
// For k <= 6 it takes less than 1 minute

#include <stdlib.h>
#include <stdio.h>
#include <limits.h>
#include <time.h>

#define MAX_N  10 // The maximal number of diameters (cann't be less than k+3)

// Global array of statistics
int statistics[4*MAX_N][2*MAX_N+2]; 
// statistics[v][...] is the Gale diagram D with v points (v is the sum of labels)
// The diagram D has the minimal number of facets for the given v
// statistics[v][0] is the number of facets
// statistics[v][1] is the number of diameters 
// statistics[v][2 ... 2*ndiam+1] are the values of labels

// The number of cofacets for the wheel
int cofacets(int *wheel, int ndiam){
    int nf = 0; // The number of cofacets
	for (int i = 0; i < ndiam-1; i++){
        // Count cofacets with two points
		nf += wheel[i] * wheel[ndiam + i]; 
        // Count cofacets with three points
	    for (int j = i + 2; j < ndiam + i; j++)
			for (int k = ndiam + i + 1; k < ndiam + j && k < 2*ndiam; k++)
				nf += wheel[i] * wheel[j] * wheel[k]; 
	}
	return nf + wheel[ndiam-1] * wheel[ndiam+ndiam-1];
}

// Test of k-neighborlyness and collect statistics
// If the wheel is k-neighborly, then return 1
bool test_wheel(int *wheel, int k, int ndiam, int vertices){
	// Testing of k-neighborlyness
	int half = 0;
	for (int j = 0; j <= ndiam-2; j++)
		half += wheel[j];
    int min_half = k+1;
	for (int j = 0; j < 2*ndiam; j++){
		if (half < min_half)
			return false;
		half -= wheel[j];
		half += wheel[(j + ndiam - 1)%(2*ndiam)];	
	}	
	// Collect statistics
	int f = cofacets(wheel, ndiam);
	if (statistics[vertices][0] > f){
		// Memorize the diagram with the minimum number of cofacets
		statistics[vertices][0] = f;
		statistics[vertices][1] = ndiam;
		for (int j = 0; j < 2*ndiam; j++) 
			statistics[vertices][j+2] = wheel[j];
	}
	return true;
}

// Recursive procedure for enumerating weehls
// 'i' is the index of the current label
void step(int k, int ndiam, int *wheel, int min_label_value, int max_label_value, int i, int sum_of_labels){
	int label = min_label_value; // The value of the current label wheel[i]
    // Sum of two consequtive labels must be >= 2
	if (label < 2 - wheel[i-1]) 
		label = 2 - wheel[i-1];
	sum_of_labels += label;
	// Look over all possible values
	for ( ; label <= max_label_value && sum_of_labels <= 4*(k+1); label++, sum_of_labels++){
		wheel[i] = label;
		if (i == 2*ndiam-1){
			if (test_wheel(wheel, k, ndiam, sum_of_labels)) 
				break; // Stop for the minimum possible value of wheel[2*ndiam-1]
		}	
		else
			step(k, ndiam, wheel, min_label_value, max_label_value, i+1, sum_of_labels);
	}
}

// Print statistics into the file k-neighborly.txt
void print_statistics(int k){
    // Open file
	char filename[32]; // file name
	sprintf (filename, "%d-neighborly.txt", k);
	FILE *out = fopen(filename, "w");
	fprintf (out, "Look for extremal %d-neighborly reduced Gale diagrams (without O-label)\n\n", k);
	fprintf(out, "Minimal number of facets (f) for a fixed number of vertices (v):\n");
	fprintf(out, "  v,   f,  f-v, diam|   labels\n");
	int min_dif = INT_MAX; // The minimum difference: facets - vertices
	int v_min; // The number of vertices for the diagram with minimum difference
    int min_half = k+1; // Minimal sum of an open half of a k-neighborly wheel
	// Cycle through all possible numbers of vertices
	for(int vertices = 2*min_half; vertices <= 4*min_half; vertices++){
		int facets = statistics[vertices][0];
		if (facets < INT_MAX){
			int ndiam = statistics[vertices][1];
			int dif = facets - vertices;
			if (dif < min_dif){
				min_dif = dif;
				v_min = vertices;
			}
			fprintf (out, " %2d, %3d,  %3d,   %2d|  ", 
			              vertices, facets, dif, ndiam);
			for (int j = 0; j < 2*ndiam; j++) 
				fprintf (out, " %d", statistics[vertices][j+2]);
			fprintf (out, "\n");
		}
	}
	
	fprintf(out, "\nMin(facets - vertices) = %d\n", min_dif);
	fprintf(out, "The extremal diagram has %d diameters, %d vertices and %d facets\n", 
	             statistics[v_min][1], v_min, statistics[v_min][0]);
	fclose (out);
}
	
int main (int argc, char *argv[]){
	clock_t start = clock();
    int k; // The degree of neighborliness
    int wheel[2*MAX_N]; // Current Gale diagram with labels 'wheel[i]'
	for (k = 2; k <= 6; k++){
		printf ("Enumerate %d-neighborly diagrams:", k);
		int min_half = k+1; // The minimal sum of labels in an open half of a k-neighborly wheel
		// Clear statistics
		for(int v = 2*min_half; v <= 4*min_half; v++)
			statistics[v][0] = INT_MAX;
		// Enumerate all cases		
        // Cycle through the number of diameters
		for (int ndiam = 2; ndiam <= k + 3; ndiam++){ //ndiam <= k + 2 + k%2
			printf (" %d", k + 3 - ndiam);
            int max_label_value = min_half; // The maximum value of a label
            if (ndiam > 4)
                max_label_value = k + 5 - ndiam - ndiam%2;
			for (int min_label_value = 0; min_label_value <= min_half/(ndiam-1); min_label_value++){
				if (min_label_value > 0) // Correct the maximum value of a label
					max_label_value = min_half - (ndiam-2)*min_label_value;
				wheel[0] = min_label_value; // Let wheel[0] has a minimal value
                // The main recursive procedure for enumerating the k-neighborly wheels
				step(k, ndiam, wheel, min_label_value, max_label_value, 1, wheel[0]); 
			}	
		}
		printf ("\n");
		print_statistics(k);
	}
	printf ("Elapced time: %.2f sec\n", (float)(clock() - start)/CLOCKS_PER_SEC);
	return 0;
}
\end{lstlisting}  


\bigskip
\begin{thebibliography}{15}
	
	\bibitem{Firsching:2017} 
	M. Firsching,
	\emph{Realizability and inscribability for simplicial polytopes via nonlinear optimization,}
	Mathematical Programming, \textbf{166}:1--2 (2017), 273--295.
	
	
	\bibitem{Fusy:2006} 
	\'E. Fusy,
	\emph{Counting $d$-Polytopes with $d+3$ Vertices,}
	Electr. J. Comb., \textbf{13} (2006), research paper \#R23.
	
	\bibitem{Henk:2004} 
	M. Henk, J. Richter-Gebert and G. Ziegler,
	\emph{Basic properties of convex polytopes,} 
	In~J.E.~Goodman and J.~O'Rourke, editors, 
	Handbook of Discrete and Computational Geometry,
	Chapman \& Hall/CRC Press, Boca Raton, 2nd edition, 2004, 355--382. 
	
	\bibitem{Gillmann:2006} 
	R. Gillmann,
	\emph{0/1-Polytopes: Typical and Extremal Properties,}
	PhD Thesis, TU Berlin, 2006.
	
	\bibitem{Grunbaum:2003}
	B. Gr\"unbaum,
	\emph{Convex polytopes}, 2nd edition 
	(V. Kaibel, V. Klee and G.M. Ziegler, eds.),
	Springer, 2003. 
	
	\bibitem{Maksimenko:2010}
	A.N. Maksimenko,
	\emph{On the number of facets of a 2-neighborly polytope,}
	Model. Anal. Inform. Sist.,  \textbf{17}:1 (2010), 76--82 (in Russian).
	
	\bibitem{Maksimenko:2014} 
	A. Maksimenko,
	\emph{$k$-Neighborly Faces of the Boolean Quadric Polytopes,}
	Journal of Mathematical Sciences, \textbf{203}:6 (2014), 816--822.
	
	\bibitem{Maksimenko:2016} 
	A.N. Maksimenko,
	\emph{A special role of Boolean quadratic polytopes among other combinatorial polytopes,}
	Model. Anal. Inform. Sist., \textbf{23}:1 (2016), 23--40.
	
	\bibitem{Marcus:1981}
	D.A. Marcus,
	\emph{Minimal Positive 2-Spanning Sets of Vectors,}
	Proc. AMS, \textbf{82}:2 (1981), 165--172.
	
	\bibitem{McMullen:1970} 
	P. McMullen,
	\emph{The maximum numbers of faces of a convex polytope,}
	Mathematika, \textbf{17} (1970), 179--184. 
	
	
	\bibitem{Padrol:2013}
	A. Padrol,
	\emph{Many neighborly polytopes and oriented matroids,} 
	Discrete \& Computational Geometry, \textbf{50}:4 (2013), 865--902.
	
	
\end{thebibliography}
\end{document}